\documentclass[notitlepage]{amsart}

\newtheorem{pro}{Proposition}[section]
\newtheorem{teo}[pro]{Theorem}
\newtheorem{defi}[pro]{Definition}
\newtheorem{lem}[pro]{Lemma}
\newtheorem{cor}[pro]{Corollary}
\newtheorem{rk}[pro]{Remark}
\newtheorem{ex}[pro]{Example}

\newcommand{\Ext}{\mathrm{Ext}}

\newcommand{\Hom}{\mathrm{Hom}}

\newcommand{\A}{\mathcal{A}}

\newcommand{\X}{\mathcal{X}}
\newcommand{\Y}{\mathcal{Y}}
\newcommand{\Z}{\mathcal{Z}}
\newcommand{\W}{\mathcal{W}}
\newcommand{\pd}{\mathrm{pd}}
\newcommand{\Gpd}{\mathrm{Gpd}}
\newcommand{\WGpd}{\mathrm{WGpd}}

\newcommand{\Proj}{\mathrm{Proj}}

\newcommand{\Inj}{\mathrm{Inj}}

\newcommand{\id}{\mathrm{id}}

\newcommand{\resdim}{\mathrm{resdim}}

\newcommand{\coresdim}{\mathrm{coresdim}}

\newcommand{\Flat}{\mathrm{Flat}}

\newcommand{\Modu}{\mathrm{Mod}}
\newcommand{\Ker}{\mathrm{Ker}}
\newcommand{\Ima}{\mathrm{Im}}
\newcommand{\GP}{\mathcal{GP}}
\newcommand{\DI}{\mathcal{DI}}

\newcommand{\glGPD}{\mathrm{gl.GPD}}
\newcommand{\glGID}{\mathrm{gl.GID}}

\newcommand{\glWGID}{\mathrm{gl.WGID}}
\newcommand{\glWGPD}{\mathrm{gl.WGPD}}

\newcommand{\DP}{\mathcal{DP}}

\newcommand{\GI}{\mathcal{GI}}

\usepackage{latexsym,amssymb,amscd} 
\usepackage{amsmath}
\usepackage[all]{xypic}
\usepackage{amsthm}

\newcommand{\cogorro}{\vee}
\newcommand{\ortogonal}{\bot}
\newcommand{\gorro}{\wedge}

\begin{document}
\title[Relative Global Gorenstein dimensions]{Relative Global Gorenstein dimensions}
\author{V\'ictor Becerril}
\address[V. Becerril]{Instituto de Matem\'atica y Estad\'istica ``Prof. Ing. Rafael Laguardia''. Facultad de Ingenier\'ia. Universidad de la Rep\'ublica. CP11300. Montevideo, URUGUAY}
\email{vbecerril@fing.edu.uy}
\thanks{Key Words: }
\begin{abstract} 
Let  $\A$ be   an abelian category. In this paper, we investigate  the global $(\X, \Y)$-Gorenstein projective dimension $\glGPD_{(\X,\Y)}(\A)$, associated to a GP-admissible pair $(\X, \Y)$. We give homological conditions over $(\X, \Y)$ that characterize it. Moreover,  given a GI-admisible pair $(\Z, \W)$, we study conditions under which $\glGID_{(\Z,\W)}(\A) $ and $\glGPD_{(\X,\Y)}(\A)$ are the same.
\end{abstract}  
\maketitle

\section{Introduction} 

Given an abelian category $\A$, the Gorenstein projective (injective) objects over $\A$ \cite{BMS} are a generalization which unifies the Gorensteiness notions over a ring $R$. Some of them are;  \textit{Gorenstein projective} $R$-modules \cite{Holm04}  $\GP (R)$, \textit{Ding projective} $R$-modules \cite{Gill10}, $\DP(R)$ (also called \textit{strongly Gorenstein flat} \cite{DingLi}), and \textit{AC-Gorenstein} $R$-modules \cite{BGH}  $\GP _{\mathrm{AC}}(R)$. Each one of  these  classes allows us to develop a relative homological algebra in $\A$ with similar properties  to the homological algebra  that comes from projective $R$-modules. Until now, it has been proved by Bennis \cite{Benn10} for the Gorenstein projective  $R$-modules, that the \textit{left global Gorenstein projective dimension} coincides with the  \textit{left global Gorenstein injective dimension} when one of these is finite, and so either of these dimensions is denoted by $\mathrm{gl.G.dim}(R)$. Also, whenever $R$ is a Ding-Chen ring, the  \textit{left global Ding projective dimension} coincides with the  \textit{left global Ding injective dimension}  \cite[Theorem 3.11]{Yang}, provided that one of the two is finite (so denoted by $\mathrm{gl.D.dim}(R)$). 
Furthermore, for any Ding Chen ring $R$, by \cite[Theorem 3.1]{MadTam} it holds  $\mathrm{gl.G.dim}(R) = \mathrm{gl.D.dim}(R)$.  However,  it is not known whether or not, the equality beetwen these dimensions  holds for any ring (see \cite[\S 3.1]{Wang}).  We answer partially this question in  Proposition \ref{Question}. Also for a Ding Cheng ring $R$, was proved in \cite[Corollary 3.5]{DingLi} that $\mathrm{gl.D.dim}(R)$ is finte if and only if $\sup \{\id_R (F) : F \mbox{ is any flat } R\mbox{-module}\}$ is finite. We prove in Proposition \ref{ikenaga}  that there exist a relation between the injective dimension of the testing class (see \cite[Definition 3.2]{BMS}) and the global Gorenstein projective dimension of abelian category $\A$. Moreover we prove   in Proposition \ref{n-Gorenstein} that over a $n$-Gorenstein ring $R$ the left global Ding projective dimension is characterized by  $\id (_R R)$. Finally, in Proposition \ref{ETriangulada}, we see that under certain conditions, there exist a hereditary and complete cotorsion triple $(\GP _{(\X, \Y)}, \Inj (\A) ^{\cogorro} _n , \GI _{(\Z, \W)}) $, which recovers  Chen's  \cite[Lemma 5.1]{Chen}, and that there is a triangle equivalence between the chain homotopy categories $\mathbf{K}(\GP _{(\X, \Y)})  \cong \mathbf{K} (\GI _{(\Z, \W)})$.

\section{Preliminaries}
Throughout this paper, $\A$ will be an abelian category. We denote by $\pd (A)$ the \textbf{projective dimension} of $A \in \A$, and by $\Proj(\A)$ the class of all the objects  $A \in \A$ with $\pd (A)=0$. Similarly, $\id (A) $ stands for the \textbf{injective dimension} of $A \in \A$, and $\Inj(\A)$ for the class of all the objects $A \in \A$ with $\id (A) =0$.

 Let $\X $ be  a class of objects in $\A$ and $M \in \A$. We set the following notation:
 \textit{Orthogonal classes}.  For each positive integer $i$, we consider the right orthogonal classes 
$$\X^{\ortogonal _i} := \{N \in \A: \Ext ^{i} _{\A} (-,N) | _{\X} =0\} \mbox{ and } \X ^{\ortogonal} := \cap _{i >0} \X^{\ortogonal _i}.$$
Dually, we have the left orthogonal classes $^{\ortogonal _i} \X$ and $^{\ortogonal} \X$.

Given a class $\Y\subseteq \A$, we write $\X \ortogonal \Y$ whenever $\Ext ^{i} _{\A} (X, Y) =0$ for all $X \in \Y$,  $Y \in \Y$ and $i >0$.  

 \textit{Relative homological dimensions}. The \textbf{relative projective dimension of $M$, with respect to $\X$}, is defined as
$$\pd _{\X} (M) : = \min \{n \in \mathbb{N}: \Ext ^{j} _{\A} (M,-) | _{\X} =0 \mbox{ for all } j>n\}.$$
We set $\min \emptyset := \infty$. Dually, we denote by $\id _{\X} (M)$ the \textbf{relative injective dimension of $M$ with respect to $\X$}. Furthermore,  we set  
$$\pd _{\X} (\Y) := \sup\{ \pd _{\X} (Y): Y \in \Y\} \mbox{ and } \id _{\X} (\Y) := \sup \{\id _{\X} (Y): Y \in \Y\}. $$
If $\X = \A$, we just write $\pd (\Y)$ and $\id (\Y)$.

 \textit{Resolution and coresolution dimension}. The \textbf{$\X$-coresolution dimension of $M$} denoted $\coresdim _{\X} (M)$, is the smallest non-negative integer $n$ such that there is an exact sequence 
$$0 \to M \to X_0 \to X_1 \to \cdots \to X_n \to 0,$$ 
with $X_i \in \X$ for all $i \in \{0, \dots , n\}$. If such $n$ does not exist, we set $\coresdim _{\X} (M) := \infty$. Also, we denote by $\X^{\cogorro } _n$ the class of objects in $\A$ with $\X$-coresolution dimension at most $n$. The union $\X ^{\cogorro} := \cup _{n \geq 0} \X ^{\cogorro} _n$ is the class of objects in $\A$ with finite $\X$-coresolution dimension.
Dually, we have the $\X$-{\bf{resolution dimension}} $\resdim_\X\,(M)$ of $M,$  $\mathcal{X}^{\wedge} _n$  the class of objects in 
$\A$ having  $\mathcal{X}$-resolution dimension at most $n$, and the union $\X ^{\gorro} := \cup _{n \geq 0} \X ^{\gorro} _n$ is the class of objects in $\A$ with finite $\X$-resolution dimension. We set  
$$\coresdim_\X\,(\Y):=\mathrm{sup}\,\{\coresdim_\X\,(Y)\;
:\; Y\in\Y\},$$ and  $\resdim_\X\,(\Y)$ is defined dually.

 \textit{Relative Gorenstein objects}. Given a pair $(\X,\Y)$ of classes of objects in $\A$, an object $M \in \A$ is:
 \begin{enumerate}
\item \textbf{$(\X,\Y)$-Gorenstein projective} \cite[Definition 3.2]{BMS} if $M$ is a cycle of an exact complex $X_\bullet$ with $X_m \in \X$ for every $m \in \mathbb{Z}$, such that the complex $\Hom_{\A}(X_{\bullet},Y)$ is an exact complex for all $Y \in \Y$. The class of all $(\X, \Y)$-Gorenstein projective objects is denoted by $\GP _{(\X, \Y)}$.  

\item \textbf{weak $(\X,\Y)$-Gorenstein projective} \cite[Definition 3.11]{BMS} if $M \in {^{\ortogonal} \Y} $ and there is an exact sequence $0 \to M \to X^{0} \to X^1 \to \cdots$, with $X^{i} \in \X$ and $\Ima (X^{i} \to X^{i+1})  \in {^{\ortogonal} \Y}$ for all $i \in \mathbb{N}$.  The class of all weak $(\X,\Y)$-Gorenstein projective objects is denoted by $W\GP _{(\X, \Y)}$.
\end{enumerate}
 The  \textbf{$(\X,\Y)$-Gorenstein  projective dimension of $M$} is defined by 
$$\Gpd_{(\X,\Y)}(M):=\resdim_{\GP_{(\X,\Y)}}(M),$$
and for  any class $\mathcal{Z}\subseteq\A,$   $\Gpd_{(\X,\Y)}(\mathcal{Z}):=\sup\{\Gpd_{(\X,\Y)}(Z)\;:\;Z\in\mathcal{Z}\}.$ \\  
The \textbf{weak $(\X,\Y)$-Gorenstein  projective dimension of $M$} is defined by 
$$W\Gpd_{(\X,\Y)}(M):=\resdim_{W\GP_{(\X,\Y)}}(M),$$
and for  any class $\mathcal{Z}\subseteq\A,$  
$$W\Gpd_{(\X,\Y)}(\mathcal{Z}):=\sup\{W\Gpd_{(\X,\Y)}(Z)\;:\;Z\in\mathcal{Z}\}.$$
Dually, we have the notion of (resp. weak) \textbf{$(\X,\Y)$-Gorenstein injective objects}, and the class of all them  is denoted by (resp. $W\GI_{(\X, \Y)}$)  $\GI_{(\X, \Y)}$.  

We recall some notions and results from \cite{BMS}.
\begin{defi}\cite[Definition 4.17]{BMS} For any pair $(\X,\Y)$ of classes of objects in an abelian category $\A,$ we consider the following global homological dimensions. 
\begin{itemize}
\item[(1)] The {\bf global  $(\X,\Y)$-Gorenstein projective dimension} of $\A$ 
\begin{center} $\glGPD_{(\X,\Y)}(\A):=\Gpd_{(\X,\Y)}(\A).$\end{center}
\item[(2)] The {\bf weak global  $(\X,\Y)$-Gorenstein projective dimension} of $\A$  
\begin{center}$\glWGPD_{(\X,\Y)}(\A):=\WGpd_{(\X,\Y)}(\A).$\end{center}
\end{itemize}
Similarly, we have the {\bf global  $(\X,\Y)$-Gorenstein injective dimension}  of $\A,$ denoted
$\glGID_{(\X,\Y)}(\A)$,   and the {\bf weak global  $(\X,\Y)$-Gorenstein injective dimension} of $\A$, 
$\glWGID_{(\X,\Y)}(\A)$. 
\end{defi}
\begin{defi}\cite[Definition 3.1]{BMS}
A pair $(\X, \Y) \subseteq \A$ is \textbf{GP-admissible} if for each $A \in \A$, there is an epimorphism $X \to A$, with $X \in \A$, and $(\X, \Y)$ satisfies the following two conditions:
\begin{itemize}
\item[(a)] $\X$ and $\Y$ are closed under finite coproducts in $\A$, and $\X$ is closed under extensions;
\item[(b)]   for all $X \in \A$ there exist a short exact sequence $0 \to X \to W \to X' \to 0$, with $W \in \X \cap \Y$, $X' \in \X$, and $\X \ortogonal \Y$.
\end{itemize}
\end{defi}

GI-admissible pairs are defined dually (see \cite[Definition 3.6]{BMS}). \\

Throughout this paper, a  GP-admissible pair $(\X, \Y)$ in $\A$ will always  be such that $ \X \cap \Y = \Proj (\A)$.  Similarly 
a  GI-admissible pair $(\Z, \W)$ in $\A$ will always  be such that $ \Z \cap \W = \Inj (\A)$. 
Note that when $\A$ has enough projectives, the pair $(\Proj(\A), \Proj(\A))$ is GP-admissible, and if $\A$ has enough injectives, the pair $(\Inj (\A), \Inj(\A))$ is GI-admissible.
\begin{ex} \label{Ejem}
Let $R$ be an arbitrary ring.
\begin{enumerate}
\item Recall that an $R$-module $M$ is called \textbf{FP-injective} \cite{Gill10} if $\Ext ^{1} _{R} (N,M) =0$ for all finitely presented $R$-modules $N$. Denote by $\mathcal{FP}\Inj (R)$ the class of all FP-injective $R$-modules, and by $\Flat (R)$ the class of all flat $R$-modules. The pair $(\Proj (R), \mathrm{Flat} (R))$ is GP-admissible, the pair  $(\mathcal{FP}\Inj (R)(R), \Inj (R))$ is GI-admissible, and the classes $\GP _{(\Proj (R), \mathrm{Flat} (R))}$ and $\GI_{(\mathcal{FP}\Inj (R)(R), \Inj (R))}$ are called \textbf{Ding projective modules} and \textbf{Ding injective modules},  respectively (also denoted by $\DP(R)$ and $\DI(R)$, respectively).
\item An $R$-module $F$ is of \textbf{type FP$_{\infty}$}  \cite{BGH} if it has a projective resolution $\cdots \to P_1 \to P_0 \to F \to 0$, where each $P_i$ is finitely generated.  An $R$-module $A$ is \textbf{absolutely clean} if $\Ext ^{1} _{R} (F,A) =0$ for all $R$-modules $F$ of type FP$_{\infty}$. Also an $R$-module $L$ is \textbf{level} if $\mathrm{Tor} _{1} ^{R} (F,L) =0$ for all (right) $R$-module $F$ of type FP$_{\infty}$. We denote by $\mathrm{AC} (R)$  the class  of all absolutely clean $R$-modules, and by  $\mathrm{Lev} (R)$ the class of all level $R$-modules. The pair $(\Proj (R), \mathrm{Lev} (R))$ is GP-admissible and the pair  $(\mathrm{AC}(R), \Inj (R))$ is GI-admissible. The classes of $R$-modules $\GP _{(\Proj (R), \mathrm{Lev} (R))}$ and $\GI_{(\mathrm{AC}(R), \Inj (R))}$ are called \textbf{Gorenstein AC-projective} and \textbf{Gorenstein AC-injective}, respectively.  
\end{enumerate}
\end{ex}

\begin{rk}\label{igualdad}\
Let $(\X, \Y)$ be a GP-admissible pair in an abelian category $\A$. From \cite[Corollary 4.15 (b2) and Corollary 4.18 (a)]{BMS}, we have that   the following equalities are true for all $M \in \Proj (\A) ^{\gorro}$
$$\Gpd _{(\X, \Y)} (M) = \resdim _{\Proj (\A)} (M) = \pd _{\Proj (\A)} (M) = \pd (M).$$
\end{rk}

By the previous remark, we can rewrite \cite[Theorem 4.1]{BMS} changing the resolution dimension by the Ext-dimension, as follows.

\begin{teo}\label{GPAB4} Let $(\X,\Y)$ be a GP-admissible pair in $\A$. Then,  for any  $C\in\A$ with $\Gpd_{(\X,\Y)}(C)=n<\infty$, there is an exact sequence in $\A,$ 
\begin{center}
$0\to K\to G\xrightarrow{\varphi}C\to 0,\;$ 
\end{center}
 with  $G\in \GP_{(\X,\Y)}$ where $\varphi:G\to C$ is an $\GP_{(\X,\Y)}$-precover, and  $\pd (K)= n-1$.
\end{teo}

\section{Global dimensions}
In this section we develop the necessary tools  to characterize the finitude of the relative global dimensions from the homological properties of the GP and GI admissible pairs, and we will see how these results have applications to the Gorensteiness notions over the left $R$-modules $\Modu(R)$ presented above. Monomorphism and epimorphism in $\A$ may sometimes be denoted using arrows $\hookrightarrow$ and $\twoheadrightarrow$, respectively.
\begin{lem}\label{finita}
Let  $(\X, \Y)$ be GP-admissible and  $(\Z, \W)$ be GI-admissible in  $\A$. Then, $\GP _{(\X, \Y)} \ortogonal \Inj (\A) ^{\cogorro}$. 
\end{lem}

\begin{proof}
Indeed, take  $N \in \Inj (\A) ^{\cogorro}$ and $G \in \GP _{(\X, \Y)}$. Since $(\X, \Y)$ is GP-admissible,  by \cite[Corolary 3.25 (a)]{BMS} 
for each $n \geq 0$ there is an  exact sequence 
$$0 \to G_n \to P_n \to G_{n+1} \to 0$$  with $G_i \in \GP _{(\X, \Y)}$, $P_i \in \Proj (\A)$ and  $G_0 := G$. Applying $\Hom _{\A} (-,N)$, we have the exact sequence 
$$\Ext _{\A} ^{i} (P_n ,N) \to \Ext _{\A} ^{i} (G_n , N) \to \Ext _{\A} ^{i+1}  (G_{n+1}, N)\to \Ext _{\A} ^{i+1} (P_n,N),$$
with zero ends. Thus $\Ext _{\A} ^{i}  (G, N) \cong \Ext _{\A} ^{i+n} (G_n , N)$ for all $i > 0$ and $n\geq 0$. From  Remark \ref{igualdad} (ii), we have that  $\id (N) = \coresdim _{\Inj(\A)} (N) <\infty$. Then for $n+i > \coresdim _{\Inj (\A)} (N)$, we have $\Ext _{\A} ^{i+n} (G_n, N )=0 $, i.e $\Ext _{\A} ^{i} (G,N) =0 $ for all $i > 0$.
\end{proof}
A pair $(\X, \Y) \subseteq \A^{2} $ is a \textbf{cotorsion pair} if $\X^{\ortogonal _1} = \Y$ and $\X = {^{\ortogonal _1} \Y}$. This cotorsion pair is  \textbf{complete} if for any $A \in \A$, there are exact sequences $0 \to Y \to X \to A \to 0$ and $0 \to A \to Y' \to X' \to 0$ where $X, X' \in \X $ and $Y, Y' \in \Y$. Moreover, pair is \textbf{hereditary} if $\X \ortogonal \Y$. A triple $(\X,\Y,\Z) \subseteq \A ^{3}$ is called a \textbf{cotorsion triple} \cite{Chen} provided that both $(\X, \Y)$ and $(\Y, \Z)$ are cotorsion pairs; it is \textbf{complete} (resp. \textbf{hereditary}) provided that both of the two cotorsion  pairs are complete (resp. hereditary).

\begin{pro} \label{parinyectivo}
Let  $(\X, \Y)$ be GP-admissible and  $(\Z, \W)$ be GI-admissible in  $\A$. If $(\GP _{(\X, \Y)}, \Inj (\A) ^{\cogorro} _n )$ is a  complete cotorsion pair, then $\glGPD_{(\X,\Y)}(\A)\leq n $.
\end{pro}

\begin{proof}
Take $A \in \A$. Since $(\GP _{(\X, \Y)}, \Inj (\A) ^{\cogorro} _n )$ is complete, there is an exact sequence  $$0 \to K_n \to G_{n-1} \to \cdots \to G_{0} \to A \to 0,$$
 with $G_i \in \GP _{(\X ,\Y)} $ for all $i\geq 0$ and $K_n \in \Inj (\A) ^{\cogorro} $. Now for $Q\in \Inj (\A)^{\cogorro}_n$, by Lemma \ref{finita} and dimension shifting we have
$\Ext ^1 _{\A} (K_n, Q) \cong \Ext _{\A} ^{n+1} (A, Q) =0,$
where the last term is zero by Remark \ref{igualdad} (ii). Thus $K_n \in {^{\ortogonal_1}(\Inj (\A)^{\cogorro}_n)} = \GP _{(\X ,\Y)}.$ 
\end{proof}

\begin{lem}\label{desigualdad}
Let  $(\X, \Y)$ be GP-admissible and  $(\Z, \W)$ be GI-admissible in  $\A$. Then, for all $M \in \Inj (\A) ^{\cogorro}$, it holds    $\pd (M) \leq   \Gpd _{(\X, \Y)} (M) $. 	
\end{lem}

\begin{proof}
If $\Gpd _{(\X, \Y)}(M) = \infty$ there is nothing to prove. Assume that $\Gpd _{(\X, \Y)}(M) =n < \infty$. Thus from Theorem \ref{GPAB4}, there is a exact sequence $0 \to K \to G \to M \to 0$, with $G \in \GP_{(\X, \Y)} $ and $ \pd (K)= \resdim _{\Proj(\A)} (K) =n-1$. We know that there is an exact sequence $0 \to  G \to P \to G' \to 0$ with $P \in \Proj (\A)$ and $G \in \GP _{(\X, \Y)}$, thus we have the following exact  and commutative  pushout diagram 
$$\xymatrix{
 K\;  \ar@{^{(}->}[r] \ar@{=}[d] & G_{} \ar@{^{(}->}[d]  \ar @{} [dr] |{\textbf{po}}\ar@{>>}[r] & M_{} \ar@{^{(}->}[d]  \\
 K\; \ar@{^{(}->}[r] & P  \ar@{>>}[d] \ar@{>>}[r]& L  \ar@{>>}[d] \\
  & G'  \ar@{=}[r]  &  G'    }$$
Since $P \in \Proj (\A)$ and $\resdim _{\Proj(\A)}(K) =n-1$ it follows that $\resdim _{\Proj(\A)} (L )\leq n$, in particular by Remark \ref{igualdad} (i), $\pd (L) \leq n$. Now by  Lemma \ref{finita} $\Ext ^{1} _{\A} (G' , M) =0$, we have that $0 \to M \to L \to G' \to 0$ splits, thus $\pd (M) \leq \pd(L) \leq n$.
  \end{proof} 
 
We will use the following result frequently.

\begin{teo}\label{Thm4.22}
Let  $(\X, \Y)$ be GP-admissible and  $(\Z, \W)$ be  GI-admissible in  $\A$.  The following statements are true.
\begin{itemize}
\item[(i)] $ \pd (\Inj (\A)) = \pd (\Inj (\A) ^{\cogorro} ) \leq \glGPD_{(\X,\Y)}(\A)$, 
\item[(ii)] If $\GP _{(\X,\Y)} ^{\gorro} = \A$, then $\id (\Proj (\A)) = \id (\Y) = \glGPD_{(\X,\Y)}(\A).$
\end{itemize}
\end{teo}

\begin{proof}
(i) The first equality is \cite[Lemma 2.13]{MS}, and the second follows from Lemma \ref{desigualdad}.

(ii) Since $(\X, \Y)$ is GP-admissible  $W\GP_{(\Proj (\A), \Y)} = \GP _{(\X, \Y)} $ by \cite[Theorem 3.32]{BMS}. This implies that  $\glGPD_{(\X,\Y)}(\A) = \glWGPD_{(\Proj (\A),\Y)}(\A).$
Furthermore, by \cite[Remark 4.7]{BMS}  the pair $(\Proj (\A), \Y)$ is WGP-admissible. Then by \cite[Theorem 4.22 (g)]{BMS}, since $\GP _{(\X,\Y)} ^{\gorro} = \A$,  and the previous equality, we have $\id( \Proj (\A)) = \glGPD_{(\X,\Y)}(\A)$. Now, since $\Proj (\A) \subseteq \Y$,  $\id (\Proj (\A)) \leq \id (\Y)$. Thus we only need to prove that $\id (\Y) \leq \glGPD_{(\X,\Y)}(\A)$. If $\glGPD_{(\X,\Y)}(\A) = \infty$, there is nothing to prove. Let us suppose that $n := \glGPD_{(\X,\Y)}(\A)< \infty$. Thus  $\Gpd _{(\X, \Y)}(A) \leq n$, for all $A \in \A$.  We also know that $\GP _{(\X, \Y)} \ortogonal \Y$. Then  by dimension shifting it follows that $\Ext ^{> n} _{\A} (A, Y) =0$ for all $Y \in \Y$, that is $\id (\Y) \leq n$.
\end{proof}

In the following result we conclude the same that \cite[Theorem 3.11]{Yang}, for $\X= \Proj(R)$, $\Y = \Flat (R)$ and $\Z = \mathrm{AC}(R)$, $\W = \Inj(A)$, but with different hypotheses.
\begin{cor} \label{dimensiongloblal}
Let  $(\X, \Y)$ be GP-admissible and  $(\Z, \W)$ be GI-admissible in $\A$, such that $\GP _{(\X,\Y)} ^{\gorro} = \A=\GI _{(\Z, \W)} ^{\cogorro}$. Then $\glGPD_{(\X,\Y)}(\A) = \glGID_{(\Z,\W)}(\A) $.
\end{cor}

\begin{proof}
Follows from Theorem \ref{Thm4.22} and its dual.
\end{proof}

\begin{lem} \label{complejoexacto}
Let   $\A$ be an abelian category  with enough injective objects and $\Y \subseteq \A$ such that $ \id (\Y) < \infty$.
Then, every exact complex  $\mathbf{P}$ of projective objects is $\Hom _{\A} (-, \Y)$-exact.
\end{lem}

\begin{proof}
Let $ m:=\id (\Y) < \infty $ and $\mathbf{P} : \cdots \to P_1 \to P_0 \to P^0 \to P^1 \to \cdots$ an exact complex of projective objets in $\A$. Take $Y \in \Y$. Since $\A$ has enough injective objects, $\id (Y) = \coresdim _{\Inj (\A)} (Y)$. Thus there is an exact sequence 
$$0 \to Y \to I^0  \to \cdots \to I^m \to 0,$$
 with $I^j \in \Inj (\A)$ for all $j \in \{ 0,\dots ,m\}$. Set $K^0 := Y$, $K^m := I^m$ and $K^j := \Ker (I^j \to I^{j+1}) $ for all $j \in \{1, \dots , m-1\}$, this induces for each $j \geq 0$ an exact sequence of complexes  
\[0 \to \Hom _{\A} (\mathbf{P}, K^j) \to \Hom _{\A} (\mathbf{P}, I^j) \to \Hom _{\A} (\mathbf{P}, K^{j+1}) \to 0. \]
Since $K^m = I^m$, $\Hom _{\A} (\mathbf{P}, K^{m-1}) $ is exact, and so iterating of this way we conclude that $\Hom _{\A} (\mathbf{P}, Y) $ is exact. 
\end{proof}

\begin{pro} \label{ikenaga}
Let $(\X, \Y)$ be a GP-admissible pair in $\A$, with enough projective and injective objects.  The following statements are true.
\begin{itemize}
\item[(i)] $\glGPD_{(\X,\Y)}(\A)\leq n <\infty$ if and only if $\pd (\Inj (\A))\leq n $ and $ \id (\Y) < \infty$.
\item[(ii)] $\glGPD_{(\X,\Y)}(\A)=  n $ if and only if $\pd (\Inj (\A))\leq n \leq  \id (\Y) < \infty$.
\end{itemize}
\end{pro}

\begin{proof}
(i) $(\Leftarrow)$.  Assume that $\pd (\Inj (\A))\leq n$ and $\id (\Y) < \infty $. Take $A \in \A$. We will show that $A \in (\GP_{(\X, \Y)} )^{\gorro} _n$. Since $\A$ have enough injective objects, we can construct an exact sequence  
$$0 \to A \to I^0 \to I^1 \to \cdots ,$$
 with $I^j  \in \Inj (\A)$  for all $j \geq 0$. Also, since $\A$ have enough projective objects, $\pd (I^j) = \resdim _{\Proj (\A)} (I^j)$, thus from the inequality  $\pd (\Inj (\A))\leq n$ and \cite[Ch. XVII, \S 1 Proposition 1.3]{Cartan} we can construct the following commutative diagram
$$\xymatrix{  
 0\ar[r] & Q_{_{}} \ar@{^{(}->}[d] \ar@{^{(}->}[r] &  P^0 _{n_{}}  \ar@{^{(}->}[d] \ar[r]& P^1 _{n_{}}  \ar@{^{(}->}[d] \ar[r] & \cdots \\
0 \ar[r]& Q_{n-1} \ar[r] \ar[d] & P^0 _{n-1} \ar[d]  \ar[r]&  P^1 _{n-1} \ar[d]  \ar[r] & \cdots  \\
    &  \vdots  \ar[d] & \vdots  \ar[d] & \vdots \ar[d] & \\
 0\ar[r] & Q_0 \ar@{>>}[d] \ar[r] & P^0 _0  \ar@{>>}[d] \ar[r]& P^1 _0  \ar@{>>}[d] \ar[r] & \cdots \\
0 \ar[r]& A \ar@{^{(}->}[r]  & I^0   \ar[r] &  I^1 \ar[r]  & \cdots  }$$
where  $P_i ^j \in \Proj (\A)$ for all $i \in \{0,1, \cdots , n\}$ and $j\geq 0$. With  $Q _i := \Ker (P^0 _i \to P^1 _i)\in \Proj (\A) $ for all $i \in \{0 , \cdots n-1\}$ and 
\begin{center}
$0 \to Q \to P_n ^0 \to P_n ^1 \to \cdots$\\
 $0 \to Q \to Q_{n-1} \to \cdots \to Q_0 \to A \to 0$
 \end{center}
  exact complexes. We know by Lemma \ref{complejoexacto} that every exact complex of projective objects is $\Hom _{\A} (-,\Y)$-exact. Since $\A$ have enough projective objects and $\Proj (\A) \subseteq \X$,  then  $Q $ possesses a complete $(\X, \Y)$-resolution (see \cite[Definition 3.2]{BMS}), that is $Q\in \GP_{(\X, \Y)}$,  and therefore $\Gpd (A) \leq n$. 
  
 (i) $(\Rightarrow)$ Follows from Theorem \ref{Thm4.22}.
 
 (ii) $(\Rightarrow)$ Assume $ \glGPD_{(\X,\Y)}(\A)=n < \infty$. From (i) only need prove that $n \leq \id (\Y)$, but  from Theorem \ref{Thm4.22} (ii) $\id (\Y) = \glGPD_{(\X,\Y)}(\A) = n$.
 
 $(\Leftarrow)$  If $\pd (\Inj (\A))\leq n \leq  \id (\Y) < \infty$, then from (i) we have the inequality  $ \glGPD_{(\X,\Y)}(\A)\leq n$. Thus from Theorem \ref{Thm4.22} (ii) we have $\id (\Y) = \glGPD_{(\X,\Y)}(\A)$ i.e $n \leq \glGPD_{(\X,\Y)}(\A)$.
\end{proof}

\begin{cor}
Let $(\X, \Y)$ be a GP-admissible pair in  $\A$, with enough projective and injective objects. If  $\pd (\Inj (\A)) =n $ and $\id (\Y) < \infty$ then  $ \glGPD_{(\X,\Y)}(\A)=n$.
\end{cor}

\begin{proof}
From Proposition \ref{ikenaga} (i) we have $\glGPD_{(\X,\Y)}(\A) \leq n$. Now from \cite[Lemma 2.13]{MS} and Theorem \ref{Thm4.22} (i) we have $$n = \pd (\Inj (\A)) = \pd (\Inj (\A) ^{\cogorro}) \leq \glGPD_{(\X,\Y)}(\A).$$
\end{proof}
We will give some applications of the previous results. 
Recall that a ring $R$ is called $n$\textbf{-Gorenstein} if $R$ is a left and right Noetherian with $\id (_RR) \leq n$ and $\id (R_R) \leq n$, for a non-negative integer $n$.
\begin{pro}\label{n-Gorenstein}
Let $R$ be any $n$-Gorenstein ring with $\id (\Flat (R)) < \infty$. Then,  the left global Ding projective dimension is less or equal than $n$,  and coincides with $\id (_RR)$.
\end{pro}
\begin{proof}
From \cite[Theorem 3.19]{DingChen} we have that $ \id (_RR) =\id (\Proj (R)) =   \pd (\Inj (R))$. Thus the result it  follows from Proposition \ref{ikenaga} (i) and Theorem \ref{Thm4.22} (ii), for the GP-admissible pair $(\Proj(R), \Flat (R))$.
\end{proof}
The \textbf{FP-injective dimension of} $M$, denoted $\mathrm{FP-id} (M)$ \cite{Strom}, is defined to be the smallest integer $n \geq 0 $ such that $\Ext ^{n+1} _R (N,M) =0$ for all finitely presented $R$-modules $N$. If no such $n$ exists, then we set $\mathrm{FP-id} (M) := \infty$. A ring $R$ is called \textbf{Ding-Chen}  \cite{Gill10}, when is two-sided coherent with $\mathrm{FP-id} (_R R) < \infty$ and $\mathrm{FP-id} (R _R) < \infty$. To simplify the notation we will write $\mathrm{gl.DP}(R)$ for $\mathrm{glGPD}_{(\Proj (R), \Flat (R) )} (R)$ and $\mathrm{gl.DI}(R)$ for $\mathrm{glGID}_{(\mathcal{FP}\Inj (R), \Inj (R) )}(R)$. When $ R$ is Ding-Chen, it was proved in \cite[Corollary 3.5]{DingLi} that $\mathrm{gl.DP} (R) < \infty$ if and only if $\id (\Flat (R)) < \infty$ (see \cite[\S 3.1]{Wang}).  
\begin{pro}\label{DingGlobal}
Let $R$ be a  ring. Then,  $\id (\Flat (R)) < \infty$ and $\pd (\Inj(R)) < \infty$, if and only if $\mathrm{gl.DP} (R)< \infty$.
\end{pro}
\begin{proof}
Apply Proposition \ref{ikenaga} (i) to the GP-admissible pair $(\Proj (R), \Flat(R))$.
\end{proof}
By \cite[Theorem 3.11]{Yang}, over a Ding-Chen ring $R$, we get the equality $\mathrm{gl.DP} (R)= \mathrm{gl.DI}(R)$ if any of them is finite. We denote these quantities by $\mathrm{gl.D.dim}(R)$.  Thus, when $R$ is Ding-Chen, $\id(\Flat (R)) < \infty$ if and only if $\mathrm{gl.D.dim}(R)< \infty$. 

\begin{pro}
Let R be a Ding-Chen ring. Then $\mathrm{gl.D.dim}(R) < \infty$ if and only if $\id (\Proj (R))$ and $\pd (\mathcal{FP}\Inj (R)) < \infty$.
 \end{pro}
 \begin{proof}
 Over a Ding-Chen ring, $\mathrm{gl.D.dim}(R)< \infty$ if and only if  $\mathrm{gl.DI} (R)< \infty$. From this, and by the dual of Proposition \ref{ikenaga} (i) applied to the GI-admissible pair $(\mathcal{FP}\Inj (R), \Inj(R))$ we obtain that; $\mathrm{gl.DI} (R)< \infty $ if and only if $\id (\Proj (R))< \infty$ and $\pd (\mathcal{FP}\Inj (R)) < \infty$. 
 \end{proof}
From Example \ref{Ejem} (2), to simplify the notation, we write $\mathrm{gl.GP_{AC}} (R) $ instead of $\mathrm{gl.GPD} _{(\Proj (R), \mathrm{Lev} (R))} (R)$ and $\mathrm{gl.GI_{AC}} (R) $ instead of $\mathrm{gl.GID} _{(\mathrm{AC}(R), \Inj (R))} (R) $.  We will write $\mathrm{gl.G.dim_{AC}} (R)$ to denote when the quantities  $\mathrm{gl.GP_{AC}} (R)$, $\mathrm{gl.GI_{AC}} (R) $ are the same.

\begin{pro}\label{ACgorenstein}
Let $R$ be an arbitrary ring and $n$ a non-negative integer. The following statements are true.
\begin{itemize}
\item[(i)]  $\mathrm{gl.GP_{AC}} (R) \leq n < \infty$ if and only if $\pd (\Inj(R)) \leq n$ and $\id (\mathrm{Lev}(R)) <\infty$.
\item[(ii)] $\mathrm{gl.GI_{AC}} (R) \leq n < \infty$ if and only if $\id (\Proj(R)) \leq n$ and $\pd (\mathrm{AC}(R)) <\infty$.
\item[(iii)] $\mathrm{gl.G.dim_{AC}} (R)  \leq n$ if and only if $\id (\mathrm{Lev}(R)), \pd (\mathrm{AC}(R)) < \infty$, and one of them is less  or equal than $n$.
\item[(iv)] In $\mathrm{(iii)}$, $R$ is an AC-Gorenstein ring in Gillespie's  sense \cite{Gill18}.
\end{itemize}
\end{pro}
\begin{proof}
\begin{itemize}
\item[(i)] Apply Proposition \ref{ikenaga} (i)  to the GP-admissible pair $(\Proj (R), \mathrm{Lev} (R))$.

\item[(ii)] Follows from the dual of Proposition \ref{ikenaga} (i) applied to the GI-admissible pair $(\mathrm{AC} (R), \Inj (R))$.

\item[(iii)] ($\Leftarrow$) Assume that $\id (\mathrm{Lev}(R)) < \infty$ and $\pd (\mathrm{AC}(R)) < \infty$. The containments  $\Proj (R) \subseteq \mathrm{Lev} (R)$, $\Inj (R) \subseteq \mathrm{AC} (R)$ give us  that 
$$\id (\Proj (R)) \leq \id(\mathrm{Lev} (R)) < \infty \mbox{ and } \pd (\Inj (R)) \leq \pd (\mathrm{AC} (R)) < \infty.$$
 From (i) and (ii), in particular we have that 
$$\GP _{(\Proj (R), \mathrm{Lev} (R))} ^{\gorro}(R) = \Modu (R) = \GI _{(\mathrm{AC}(R), \Inj (R))} ^{\cogorro} (R),$$
 then by Theorem \ref{Thm4.22} (ii) and Corolary \ref{dimensiongloblal} we have that  
$$\id (\mathrm{Lev}(R)) = \mathrm{gl.GP_{AC}} (R) = \mathrm{gl.GI_{AC}} (R)  =\pd (\mathrm{AC}(R)) . $$
Thus $\mathrm{gl.G.dim}_{\mathrm{AC}} (R) \leq n$, when both of the dimensions  $\id (\mathrm{Lev}(R))$ or  $\pd (\mathrm{AC}(R)) $ is less than or equal to $n$.
  
($\Rightarrow$)  It follows from (i) and (ii).

\item[(iv)] Follows from \cite[Corollary 3.9, Defintion 4.1]{Gill18}.
\end{itemize}
\end{proof}

\begin{pro}\label{cotorsionfinito}
Let $(\X, \Y)$ be a GP-admissible pair in $\A$.  If $\glGPD_{(\X,\Y)}(\A)\leq n <\infty$, then $(\GP _{(\X, \Y)}, \Proj(\A) ^{\gorro}_n)$ is a hereditary and complete cotorsion pair in $\A$.
\end{pro}

\begin{proof}
Since $(\X, \Y)$ is GP-admissible,  we have  that $W\GP_{(\Proj(\A), \Y)} = \GP _{(\X, \Y)} $ by \cite[Theorem 3.32]{BMS}, and from this we get  $$\glGPD_{(\X,\Y)}(\A) = \glWGPD_{(\Proj(\A),\Y)}(\A).$$ 
Also we have from \cite[Theorem 4.22 (g)]{BMS}, that 
 $$(W\GP _{(\Proj(\A), \Y)}, \Proj(\A) ^{\gorro}) = (\GP _{(\X, \Y)}, \Proj(\A) ^{\gorro})$$
  is a hereditary and complete cotorsion pair. We only need prove that $\Proj(\A) ^{\gorro} = \Proj(\A) ^{\gorro} _n$, but the equality follows from \cite[Corollary 4.18 (a)]{BMS}.
\end{proof}

\begin{pro}\label{prebalance}
Let $(\X, \Y)$ be a GP-admissible pair in  $\A$, with enough projective and injective objects.  Then,  $\glGPD_{(\X,\Y)}(\A) \leq  n <\infty$ if and only if $(\GP _{(\X, \Y)}, \Inj (\A) ^{\cogorro} _n )$ is a hereditary and complete cotorsion pair in $\A$.
\end{pro}

\begin{proof}
$(\Rightarrow)$ Let us suppose that $\glGPD_{(\X,\Y)}(\A) \leq  n <\infty$. From  Proposition \ref{cotorsionfinito} it is enough to see that $\Proj (\A) ^{\gorro} _n = \Inj (\A) ^{\cogorro} _n$. 
Note that the pair $(\Inj(\A), \Inj(\A))$ is GI-admissible in $\A$, since $\A$ have enough injective objects, thus from Theorem  \ref{Thm4.22} (ii) we know that $\pd (\Inj (\A) ^{\cogorro}) \leq n \mbox{ and } \id (\Proj (\A)) \leq n$. By using the dual of \cite[Lemma 2.13]{MS}, we conclude that $ \id (\Proj (\A) ^{\gorro}) \leq n$.

Take $T \in \Proj (\A) ^{\gorro}$, since $\A$ have enough injectives and $\id (\Proj (\A) ^{\gorro}) \leq n$,  $\id (T) = \coresdim_{\Inj (\A)} (T)$, i.e.  $T \in \Inj (\A)^{\cogorro} _n $. This is $\Proj (\A) ^{\gorro} \subseteq \Inj (\A) ^{\cogorro}_n$.
 
Now for $S \in  \Inj (\A)^{\cogorro}$, since $\A$ have enough projectives and $\pd (\Inj (\A) ^{\cogorro}) \leq n$, $\resdim _{\Proj (\A)} (S) = \pd (S) \leq n$ i.e. $S \in \Proj (A) ^{\gorro} _n$. This is $\Inj (\A) ^{\cogorro} \subseteq \Proj (\A) ^{\gorro}_n$.\\
$(\Leftarrow)$ The inequality $\glGPD_{(\X,\Y)}(\A) \leq n $ follows from Proposition  \ref{parinyectivo}, since the pair  $(\Inj(\A), \Inj(\A))$ is GI-admissible.
\end{proof}
 Recall that a class $\X \subseteq  \A$ is called \textbf{precovering} if for each $M \in \A$ there is  a homomorphism $f : X \to M$ such that $\Hom _{\A} (Z, f) : \Hom _{\A} (Z, X) \to \Hom _{\A} (Z,M)$ is surjective for any $Z \in \X$. Dually is defined a \textbf{preenveloping} class. A pair $(\X, \Y)$ of  classes of objects in  $\A$ is called \textbf{balanced} \cite[Defintion 1.1]{Chen} if $\X$ is precovering $\Y$ is preenveloping and for each $M\in \A$ there exist complexes $X^{\bullet} (M):\cdots \to X^1 \to X ^0 \to M \to 0$ and $Y^{\bullet}(M) : 0 \to M \to Y^0 \to Y^1 \to \cdots$, with $X^{i} \in \X$ and $Y^{i} \in \Y$ for all $i \geq 0$, such that $\Hom _{\A} (X^{\bullet} (M), Y)$ and $\Hom _{\A}(X , Y^{\bullet} (M))$ are exact complexes for all $X \in \X$ and $Y \in \Y$. Such balanced pair is called \textbf{admissible} if the homomorphism $f : X^{0} \to M$ is epic. From \cite[Lemma 2.1]{Chen} for all balanced pair $(\X, \Y)$ and $M,N \in \A$, there exist a natural isomorphism $H^n (\Hom _{\A} (X^{\bullet } (M),N)) \cong H^n (\Hom _{\A} (M, Y^{\bullet} (N)))$ between the $n$-th cohomologies. Finally,  for an additive category $\mathfrak{a}$, we denote by $\mathbf{K}(\mathfrak{a})$ the homotopy category of complexes in $\mathfrak{a}$. 
\begin{pro} \label{ETriangulada}
Let  $(\X, \Y)$ be GP-admissible and  $(\Z, \W)$ be GI-admisible in $\A$. If $\A$ has enough injective and projective objects,  $\mathrm{glGPD}_{(\X, \Y )}(\A)\leq  n $ and $ \pd (\Z) < \infty$. Then the following statements are true:
\begin{itemize}
\item[(i)]  $\mathrm{glGPD}_{(\X, \Y )}(\A) = \mathrm{glGID}_{(\Z, \W )}(\A)$,
\item[(ii)]  $(\GP _{(\X, \Y)}, \Inj (\A) ^{\cogorro} _n , \GI _{(\Z, \W)}) $ is a hereditary and complete cotorsion triple,
\item[(iii)] $(\GP _{(\X, \Y)} , \GI _{(\Z, \W)})$ is an admissible balanced pair,
\item[(iv)] there is a triangle equivalence $\mathbf{K}(\GP _{(\X, \Y)})  \cong \mathbf{K} (\GI _{(\Z, \W)})$.
\end{itemize}
\end{pro}

\begin{proof}
\begin{itemize}
\item[(i)] We know that $\id (\Proj (\A)) \leq n $, by Theorem \ref{Thm4.22} (ii), and by hypothesis $ \pd (\Z) < \infty$. Thus using the dual of Proposition \ref{ikenaga} (i), we have $\mathrm{glGID}_{(\Z, \W )}(\A) \leq n$. Finally by Corollary \ref{dimensiongloblal} we conclude  $\mathrm{glGID}_{(\Z, \W )}(\A) = \mathrm{glGPD}_{(\X, \Y )}(\A)$.

\item[(ii)] From  Proposition \ref{prebalance} we know that $(\GP _{(\X, \Y)}, \Inj (\A) ^{\cogorro} _n )$ is a hereditary and complete cotorsion pair in $\A$. Furthermore from the dual of Proposition \ref{cotorsionfinito} and (i), we  know that  $(\Inj (\A) ^{\cogorro}_n, \GI _{(\Z, \W)})$ is an  hereditary and complete cotorsion pair.

\item[(iii)] Follows directly from \cite[Proposition 2.6]{Chen} and (ii).

\item[(iv)] From (iii)  $(\GP _{(\X, \Y)} , \GI _{(\Z, \W)})$ is an admissible balanced pair, and by  (i) we have that $\infty > n \geq \mathrm{glGPD}_{(\X, \Y )}(\A) = \mathrm{glGID}_{(\Z, \W )}(\A) $. Thus, the result follows from \cite[Theorem 4.1]{Chen}.
\end{itemize}
\end{proof}

To simplify we will write $\mathrm{gl.GP} (R)$ for $\mathrm{glGPD}_{(\Proj (R), \Proj (R) )}(R)$ and $\mathrm{gl.GI} (R)$ for $\mathrm{glGPD}_{(\Inj (R), \Inj (R) )}(R)$.  Recall that for every ring $R$, from \cite[Theorem 1.1]{Benn10}  $\mathrm{gl.GP} (R)$ coincides with $\mathrm{gl.GI} (R)$, thus we denote this quantities by $\mathrm{gl.G.dim}(R)$. 
 Also for every ring $R$, from \cite[Theorem 3.1]{MadTam} it holds  $\mathrm{gl.G.dim}(R) = \mathrm{gl.DP}(R)$. 
Thus when $R$ is a Ding-Chen ring and $\mathrm{gl.DP} (R)$ or $ \mathrm{gl.DI}(R)$ is finite, it holds  $\mathrm{gl.G.dim}(R) = \mathrm{gl.DP}(R)= \mathrm{gl.DI}(R)$.
But it is not known if  for any  ring $R$ the quantities $\mathrm{gl.G.dim}(R) $, $ \mathrm{gl.D.dim}(R)$ are the same (see \cite[\S 3.1]{Wang}). We have the following result related.

\begin{pro}\label{Question}
Let $R$ be a ring. If $\pd (\mathcal{FP}\Inj(R)) < \infty$ and $\id(\Proj (R)) < \infty$, then $\mathrm{gl.G.dim}(R) = \mathrm{gl.D.dim}(R) < \infty$.
\end{pro}
\begin{proof}
From the dimensions  $\pd (\mathcal{FP}\Inj(R)) < \infty$,  $\id(\Proj (R)) < \infty$ and the GI-admissible pair $(\mathcal{FP}\Inj(R), \Inj(\A))$, by de dual of Proposition \ref{DingGlobal} (i) we know that $\mathrm{gl.DI}(R) < \infty$. Then from the dual of Proposition \ref{ETriangulada} (i) with the GP-admissible pair $(\Proj (R), \Proj (R))$ we get $\mathrm{gl.GP}(R) = \mathrm{gl.DI}(R) < \infty $. That is, $\mathrm{gl.G.dim}(R) = \mathrm{gl.DI}(R) $.  Since $\mathrm{gl.G.dim}(R) < \infty$, by \cite[Corollary 2.7]{Benn10} we get $\id (\Flat (R)) < \infty$. Finally from the dimensions $\id(\Flat (R)) < \infty$, $\mathrm{gl.DI}(R) < \infty$ and  for the GP-admissible pair $(\Proj (R), \Flat (R) )$  by the dual of Proposition \ref{ETriangulada} (i), we get $\mathrm{gl.DP}(R) = \mathrm{gl.DI}(R)$.
\end{proof}

\begin{pro}\label{Question2}
Let $R$ be a ring. If $\id (\mathrm{Lev}(R)) < \infty$ and $\pd(\Inj (R)) < \infty$, then 
$$\mathrm{gl.DP}(R) =\mathrm{gl.G.dim}(R) = \mathrm{gl.GP_{AC}}(R) < \infty.$$
\end{pro}
\begin{proof}
From Proposition \ref{ACgorenstein} we have $\mathrm{gl.GP_{AC}}(R) < \infty$. Furthermore by Theorem \ref{Thm4.22} (ii) we have $\mathrm{gl.GP_{AC}}(R) = \id (\Proj (\A))$. Since projective modules are flat, and flat modules are level, we get that   
$$\id (\Proj (R)) \leq \id (\Flat (R)) \leq \id (\mathrm{Lev} (R)) < \infty.$$
 Thus by Proposition \ref{DingGlobal} and Theorem \ref{Thm4.22} (ii) we have $\mathrm{gl.DP}(R) = \id (\Proj(\A))$.  Finally from from \cite[Theorem 3.1]{MadTam} we have  $\mathrm{gl.G.dim}(R) = \mathrm{gl.DP}(R)$.
\end{proof}

\textbf{Funding} The author was fully supported by a  PEDECIBA (UDELAR-MEC) Uruguay Postdoctoral Fellowship.\\

\textbf{Acknowledgements} The author thanks to professor Marco A. Per\'ez for suggestions and nice discussions about  this paper.

\end{document}